\documentclass[reqno]{amsart}

\usepackage{amsmath,amssymb,amsthm}

\usepackage{graphicx}

\usepackage{xcolor}

\usepackage{float}

\theoremstyle{plain}
\newtheorem{theorem}{Theorem}[section]
\newtheorem{lemma}[theorem]{Lemma}
\newtheorem{prop}[theorem]{Proposition}
\newtheorem{corollary}[theorem]{Corollary}
\newtheorem{definition}[theorem]{Definition}

\newtheorem{problem}[theorem]{Problem}

\numberwithin{equation}{section}

\usepackage{amsmath,amssymb,amsthm,mathrsfs}
\usepackage{longtable}
\usepackage{enumerate}
\allowdisplaybreaks

\usepackage[all]{xy}

\renewcommand{\S}{{\mathbb S}}

\newcommand{\R}{{\mathbb R}}
\newcommand{\C}{{\mathbb C}}

\begin{document}

\title[Lower bound for the Green energy on $\mathbb{S}^n$]{A lower bound for the logarithmic energy on $\mathbb{S}^2$ and for the Green energy on $\mathbb{S}^n$}

\author[C. Beltr\'an]{Carlos Beltr\'an}
\address{Departamento de Matem\'aticas, Estad\'{\i}stica y \\Computaci\'on,
Universidad de Cantabria. 39005. Santander, Spain}
\email{beltranc@unican.es}

\author[F. Lizarte]{F\'atima Lizarte}
\address{Departamento de Matem\'aticas, Estad\'{\i}stica y \\Computaci\'on,
Universidad de Cantabria. 39005. Santander, Spain}
\email{fatima.lizarte@unican.es}



\thanks{Both authors belong to the Universidad de Cantabria and are supported by Grant PID2020-113887GB-I00 funded by MCIN/ AEI /10.13039/501100011033. The second author has also been supported by Grant PRE2018-086103 funded by MCIN/AEI/10.13039/501100011033 and by ESF Investing in your future}

\begin{abstract}
We show an alternative proof of the sharpest known lower bound for the logarithmic energy on the unit sphere $\mathbb{S}^2$.
We then generalize this proof to get new lower bounds for the Green energy on the unit $n$--sphere  $\mathbb{S}^n$.
\end{abstract}

\maketitle

\section{Introduction}

\noindent We consider Smale's $7$--th problem \cite{Smale2000}, that was actually posed by Michael Shub and Stephen Smale in their search for an algorithm to explicitly generate sequences of well--conditioned polynomials (see \cite{SS93III,BEMO19, BeltranLizarte21}) and is related to the distribution of a finite number of points on the unit sphere $\mathbb{S}^2$:

\begin{problem}[Smale's 7th problem]\label{Prob7Smale}
Can one find $N$ points $p_1,\ldots, p_N\in\mathbb{S}^2$ such that
$$
\mathcal{E}(p_1,\ldots, p_N)\leqslant m_N + c\ln N,
$$
for $c$ a universal constant?
\end{problem}
\noindent  Here, $\mathcal{E}(p_1,\ldots, p_N)$ is the logarithmic energy:
$$
\mathcal{E}(p_1, \ldots, p_N)=\displaystyle\sum_{i\neq j} \ln\frac{1}{\|p_i-p_j\|},
$$
and $m_N=\min_{p_1,\ldots, p_N\in \mathbb{S}^2}\mathcal{E}(p_1,\ldots,p_N)$. The problem is well known and so are its origins that can be found in the references above as well as in \cite[Section 6.7]{BHS19} and \cite{Beltran2020}, so we omit the (quite long!) historical perspective in this note.

\medskip

Smale's 7th problem, despite the simplicity of its statement, is considered a extremely difficult question. A major obstacle is that the value of the minimal logarithmic energy on the sphere is not fully known. The current knowledge, after \cite{Wagner89,RSZ94,Dubickas96,Brauchart08,BeterminSandier18,Stefan}, is the following asymptotic expansion:
$$
m_N= \kappa N^2-\frac{1}{2}N\ln N +C_{\log} N+o(N),
$$
where $\kappa$ is the continuous energy
$$
\kappa = \frac{1}{\text{vol}(\S^2)^2}\int_{p,q\in\mathbb{S}^2} \ln\frac{1}{\|p-q\|}dp\,dq=\frac12-\ln 2,
$$
and  $C_{\log}$ is a constant whose value is not known. From \cite{Lauritsen21} and \cite{BeterminSandier18} we have
\begin{equation}\label{eq:cotasClog}
-0.0568\ldots=\ln 2- \frac34\leqslant C_{\log} \leqslant 2\ln 2+\frac12\ln\frac23+3\ln\frac{\sqrt{\pi}}{\Gamma(1/3)}=-0.0556\ldots.
\end{equation}
The upper bound has been conjectured to be an equality, and one of the most important open problems in the area is the exact computation of this constant; see \cite{BeterminSandier18}, \cite{BHS19}, \cite{BrauchartHardinSaff12} and \cite{Stefan} for context. The lower bound $\ln 2- \frac34\leqslant C_{\log}$ proved by Lauritsen in \cite[Appendix B]{Lauritsen21} follows from an argument in the real plane and then invokes a sophisticated result by B\'etermin and Sandier \cite[p. 3, Theorem 1.5]{BeterminSandier18} that relates the energy in the $2$--sphere to a certain renormalized energy in the plane, and that is a purely $2$--dimensional argument which does not seem easy to translate to higher dimensions. In these pages, we describe how the argument by Lauritsen can be directly adapted to work on the sphere $\S^2$ without the use of \cite[p. 3, Theorem 1.5]{BeterminSandier18} and moreover we also show that it can be quite straightforwardly extended to spheres of arbitrary dimension.

\section{An alternative proof for Lauritsen's lower bound $C_{\mathrm{log}}\geq\ln 2- \frac{3}{4}$}

\subsection{Mean values of the logarithmic energy in the $2$--sphere.}

Let $B(p_0,a)$ be the geodesic ball centered at $p_0\in\mathbb{S}^2$ with radius $a>0$, that is, $B(p_0,a)=\{p\in\mathbb{S}^2: d_R(p_0,p)<a\}\subseteq \mathbb{S}^2$ where $d_R(\cdot,\cdot)=\arccos\langle \cdot, \cdot \rangle$
is the Riemannian distance on the unit sphere. Let $|B_a|$ be the volume of this ball. The following result is known:

\begin{prop}\label{int_bolas_esfera}
Let $p_0, p\in\mathbb{S}^2$ and $a\in(0,\pi)$. Then, $|B_a|=4\pi\sin^2\frac{a}{2}$ and
\begin{enumerate}[i)]
\item If $p\notin B(p_0,a)$
$$
\frac{1}{|B_a|}\int_{q\in B(p_0,a)} \ln\|p-q\|dq= \ln\|p-p_0\|-\frac12-\cot^2\frac{a}{2}\ln\cos\frac{a}{2}.
$$
\item If $p\in B(p_0,a)$
$$
\frac{1}{|B_a|}\int_{q\in B(p_0,a)} \ln\|p-q\|dq=\ln 2-\frac{1}{2}-\frac{\cot^2 \frac{a}{2}}{2}\ln\left(1-\frac{\|p-p_0\|^2}{4}\right)+\ln\left(\sin\frac{a}{2}\right).
$$
\end{enumerate}
\end{prop}

\begin{proof}
This result is \cite[Proposition 3.2]{Beltran2013}. Note that in that reference, the result is given for the radius $1/2$ sphere in $\mathbb{R}^3$. The translation to the unit sphere is straightforward.
\end{proof}

\begin{lemma}\label{aux_cota_bolasnodisj}
The function $F:  (0,2)\times(0,\pi) \to  \mathbb{R}$ given by
\begin{equation*}
  F(t,\alpha)=  \ln2+\ln\sin\frac{\alpha}{2}+\displaystyle\frac{\cot^2\frac{\alpha}{2}}{2}\left(2\ln\left(\cos\frac{\alpha}{2}\right)- \ln\left(1-\frac{t^2}{4}\right)\right)-\ln t,\end{equation*}
is not negative.
\end{lemma}

\begin{proof}
It is easy to see that $F(t,\alpha)$ has a minimum in $t=2\sin\frac{\alpha}{2}$ for any fixed $\alpha\in(0,\pi)$. Indeed, $\lim_{t\to0}F(t,\alpha)=\lim_{t\to2} F(t,\alpha)=\infty$ and
\begin{align*}
\frac{\partial}{\partial t}F(t,\alpha)=t\frac{\cot^2\frac{\alpha}{2}}{4-t^2}-\frac{1}{t}, \quad \text{which equals $0$ iff} \quad t=2\sin\frac{\alpha}{2},
\end{align*}
so the minimum is at $t=2\sin\frac{\alpha}{2}$. Finally, $F(2\sin\frac{\alpha}{2},\alpha)\equiv0$ and we are done.
\end{proof}

\begin{corollary}\label{cor:cotainferiormediaS2}
  Let $p_0, p\in\mathbb{S}^2$ and $a>0$. Then,
$$
\frac{1}{|B_a|}\int_{q\in B(p_0,a)} \ln\|p-q\|dq\geq \ln\|p-p_0\|-\frac12-\cot^2 \frac{a}{2}\ln\cos\frac{a}{2},
$$
with an equality if $p\notin B(p_0,a)$.
\end{corollary}
\begin{proof}
  Immediate from Proposition \ref{int_bolas_esfera} and Lemma \ref{aux_cota_bolasnodisj}.
\end{proof}

\subsection{Proof of the lower bound of $C_{\log}$}\label{sec:alternativeproofL}

Let $p_1,\ldots, p_N\in \mathbb{S}^2$, and let
\begin{align*}
I(p_1,\ldots,p_N)=& \mathcal{E}(p_1, \ldots, p_N)\\
&+\frac{2N}{\text{vol}(\mathbb{S}^2)}\sum_{i=1}^N\int_{q\in\mathbb{S}^2}\ln\|p_i-q\|dq-\frac{N^2}{(\text{vol}(\mathbb{S}^2))^2}\int_{p,q\in\mathbb{S}^2}\ln \|p-q\|dpdq.
\end{align*}
For any $p\in\mathbb{S}^2$ we have
\begin{equation}\label{valor_esp_pot_log_S2}
\frac{1}{\text{vol}(\mathbb{S}^2)}\int_{q\in\mathbb{S}^2} \ln\|p-q\|dq=\ln 2 -\frac{1}{2}=-\kappa,
\end{equation}
and hence we immediately deduce that $I(p_1,\ldots,p_N)=\mathcal{E}(p_1, \ldots, p_N)-\kappa N^2$. Define
\begin{align*}
U_{BB} &=-\frac{N^2}{\text{vol}(\mathbb{S}^2)^2}\int_{p,q\in\mathbb{S}^2}\ln\|p-q\|dpdq \stackrel{\eqref{valor_esp_pot_log_S2}}{=}\kappa N^2,\\
U_i & = \frac{2N}{\text{vol}(\mathbb{S}^2)}\int_{\mathbb{S}^2}\ln\|p_i-q\|dq \stackrel{\eqref{valor_esp_pot_log_S2}}{=}-2\kappa N,\\
U_{ij} & = -\ln \|p_i-p_j\|,\\
\widehat{U}_i & = \frac{2N}{\text{vol}(\mathbb{S}^2)|B_a|}\int_{B(p_i,a)}\int_{\mathbb{S}^2} \ln\|p-q\|dpdq \stackrel{\eqref{valor_esp_pot_log_S2}}{=}-2\kappa N,\\
\widehat{U}_{ij} & = -\frac{1}{|B_a|^2} \int_{B(p_i,a)}\int_{B(p_j,a)} \ln\|p-q\|dpdq.
\end{align*}
Clearly,
\begin{align*}
I(p_1,\ldots,p_N)=\underbrace{U_{BB}+\displaystyle\sum_{i=1}^N \widehat{U}_i+\displaystyle\sum_{i,j}\widehat{U}_{ij}}_{(\alpha)}+ \underbrace{\displaystyle\sum_{i=1}^N (U_i-\widehat{U}_i)}_{(\beta)}\underbrace{-\displaystyle\sum_{i=1}^N \widehat{U}_{ii}}_{(\gamma)}+\underbrace{\displaystyle\sum_{i\neq j}(U_{ij} -\widehat{U}_{ij})}_{(\delta)}.
\end{align*}
We proceed to lower bound and, when possible, to calculate exactly, the terms $\alpha, \beta,  \gamma$ and $\delta$. From Proposition \ref{positividad_alpha}, and the fact that the logarithmic potential is the Green function (except multiplicative and additive constant: see \eqref{eq:GreenS2}) we deduce that $\alpha\geqslant 0$. Indeed, just take the following measure
$$
\mu(p)=\displaystyle\sum_{i=1}^N\frac{1}{|B_a|}\chi_{B(p_i,a)}(p)- \frac{N}{\text{vol}(\mathbb{S}^2)},
$$
(note that  $\int_{\mathbb{S}^2}\mu(p)dp=0$) where $\chi_A$ denotes the characteristic function of a subset $A\subset \mathbb{S}^2$, and note that
$$
\alpha = - \int_{\mathbb{S}^2}\int_{\mathbb{S}^2} \ln\|p-q\|d\mu(p)d\mu(q).
$$
\noindent Moreover, we obviously have $\beta=0$. The explicit expression of $\gamma$ is obtained immediately using Proposition \ref{int_bolas_esfera}, since the value of $\hat U_{ii}$ does not depend on the point $p_i$:
\begin{align*}
\gamma=&\frac{N}{|B_a|^2} \int_{B(p_1,a)}\int_{B(p_1,a)} \ln\|p-q\|dpdq\\
=&\frac{N}{|B_a|^2}\int_{B(p_1,a)} \left(\int_{\S^2}\ln\|p-q\|dp-\int_{B(-p_1,\pi-a)} \ln\|p-q\|dp\right)dq
\\
=&-\frac{4\pi N\kappa}{|B_a|}-\frac{N}{|B_a|^2}\int_{B(p_1,a)} \int_{B(-p_1,\pi-a)} \ln\|p-q\|dpdq
\\
=&-\frac{4\pi N\kappa}{|B_a|}-\frac{N(4\pi-|B_a|)}{|B_a|^2}\int_{B(p_1,a)} \left(
\ln\|-p_1-q\|-\frac12-\frac{\ln\cos\frac{\pi-a}{2}}{\tan^2 \frac{\pi-a}{2}}
\right)dq\\
=&-\frac{4\pi N\kappa}{|B_a|}+\frac{N(4\pi-|B_a|)}{|B_a|}\left(
\frac12+\frac{\ln\cos\frac{\pi-a}{2}}{\tan^2 \frac{\pi-a}{2}}
-\left(\ln2-\frac12-\cot^2\frac{a}{2}\ln\cos\frac{a}{2}\right)\right)
\\
=& N\left[-\kappa+\ln\sin\frac{a}{2}+\cot^2\frac{a}{2} \left(\frac12+\cot^2\frac{a}{2}\ln\cos \frac{a}{2}\right)\right].
\end{align*}

\noindent Finally, we bound $\delta$. Note that
$$
\delta=\displaystyle\sum_{i\neq j}(U_{ij}-\widehat{U}_{ij}),
$$
and applying Corollary \ref{cor:cotainferiormediaS2} twice we get for all $i,j$:
\begin{align*}
U_{ij}-\widehat{U}_{ij}\geq &-\ln\|p_i-p_j\|+\frac{1}{|B_a|}\int_{B(p_i,a)} \left(\ln\|p_j-q\|-\frac12-\cot^2\frac{a}{2}\ln\cos\frac{a}{2}\right)dq\\
\geq& -1-2\cot^2\frac{a}{2}\ln\cos\frac{a}{2}.
\end{align*}
(A similar inequality with less explicit constants has been given recently in \cite[Lemma 3.1]{MarzoMas}). Therefore, we have obtained that
\begin{multline*}
I(p_1,\ldots,p_N) \geqslant  N(N-1)\left(-1-2\cot^2\frac{a}{2}\ln\cos\frac{a}{2}\right)\\
 +N\left[\ln 2-\frac12+\ln\sin\frac{a}{2}+\cot^2\frac{a}{2}\left(\frac12+\cot^2\frac{a}{2} \ln\cos \frac{a}{2}\right)\right].
\end{multline*}
This lower bound is valid for any choice of $a$. Choosing $\sin^2\frac{a}{2}=C/N$ we have:
\begin{align*}
I(p_1,\ldots,p_N) \geqslant & N(N-1)\left(-1-\frac{N-C}{C} \ln\left(1-\frac{C}N\right)\right)-\frac12N\ln N\\
& +N\left[\ln 2-\frac12+\frac12\ln C+\frac{N-C}{C} \left(\frac12+\frac{N-C}{2C} \ln\left(1-\frac{C}N\right)\right)\right]\\
\geqslant & -\frac12N\ln N+N\frac{2\ln C-2C+4\ln 2-1}{4}-\frac{C^2}{2}+O\left(\frac{1}{N}\right),
\end{align*}
where we have used the elementary inequality
$$
-\frac{C}{N}-\frac{C^2}{2N^2}-\frac{C^3}{N^3}\leq \ln \left(1-\frac{C}{N}\right)\leq -\frac{C}{N}-\frac{C^2}{2N^2},\quad \forall N\geq 2C.
$$
Taking $C=1$ (which is the optimal value) we conclude that
$$
\mathcal{E}(p_1, \ldots, p_N)=\kappa N^2+I(p_1,\ldots,p_N)\geq \kappa N^2-\frac12N\ln N+N\left(\ln 2-\frac34\right)+o(N),
$$
proving $C_{\mathrm{log}}\geq\ln 2- \frac34$ as claimed.

\smallskip

\section{The Green function in $\mathbb{S}^n$}
How should we extend the logarithmic energy to the unit $n$--sphere $\S^n\subseteq\R^{n+1}$? In the literature it is quite frequent to study Riesz potentials and also to use the very same logarithmic potential in this context. However, in a general Riemannian manifold $\mathcal M$ the Green energy
$$
E_{\mathcal M}(p_1, \ldots, p_N)=\sum_{i\neq j}G(\mathcal M;p_i,p_j),
$$
where $G(\mathcal M;\cdot,\cdot)$ is the Green function in $\mathcal M$ associated to the Laplace--Beltrami operator, is a more natural choice since it does not depend on extrinsic quantities, and is attracting more attention in the last few years, see \cite{JuanCriadoRey}, \cite{Stefan}, \cite{Andersonetal}. It turns out that the Green function in $\S^2$ is
\begin{equation}\label{eq:GreenS2}
G(\S^2;p,q)=-\frac{1}{2\pi}\ln\|p-q\|-\frac{1}{4\pi}+\frac{\ln 2}{2\pi}.
\end{equation}
That is, up to multiplicative and additive constants, the Green energy in $\S^2$ is the logarithmic potential. We claim that the deep reason why the argument by Lauritsen can be performed directly in $\S^2$ as we have done in Section \ref{sec:alternativeproofL} is precisely this fact. Hence, an extension to the Green energy in $\S^n$ is possible.

We start by computing an explicit formula for the Green energy in $\S^n$. Let
$$
V_n=\text{vol}(\S^n)=\frac{2\pi^{\frac{n+1}2}}{\Gamma\left(\frac{n+1}2\right)}.
$$

\begin{prop}\label{prop:greenSn}
The Green function for $\mathbb{S}^n$ is $G(\mathbb{S}^n;p,q)=g(\|p-q\|)$, where
\begin{equation*}
g(t)=\frac{2}{nV_n}\displaystyle\sum_{k=0}^{\infty} \frac{(n)_k}{(k+1)\left(\frac{n}{2}+1\right)_k} \left[\left(1-\frac{t^2}{4}\right)^{k+1}- \frac{B\left(\frac{n}{2},\frac{n}{2}+k+1\right)}{B\left(\frac{n}{2},\frac{n}{2}\right)}\right].
\end{equation*}
(See \eqref{eq:pochhammer} for the definition of the Pochhammer symbol $(n)_k$).
\end{prop}

\begin{proof}
Following the method in Section \ref{AlgoritmoFGreen}, we have $G(\S^n;p,q)=\phi(\S^n;d_R(p,q))$,
\begin{equation*}
\phi(\mathbb{S}^n;r)=\int_r^{\pi}\frac{2^{n-1}B_{\cos^2\frac{s}{2}} \left(\frac{n}{2},\frac{n}{2}\right)}{V_n\sin^{n-1}s}ds+C,\text{ $C$ a constant.}
\end{equation*}
(We have also used Lemma \ref{lem:betas}). From \eqref{def_incompletebeta} and \cite[(9.131.1)]{GR15} we have
\begin{align*}
\phi(\mathbb{S}^n;r)=\frac{1}{nV_n}\int_r^{\pi} \sin s\, {}_2F_1\left(1,n;\frac{n}{2}+1;\cos^2\frac{s}{2}\right)ds+C.
\end{align*}
Using \eqref{Serie_hiper_Gauss}, the half--angle identities and the change of variables $t=\frac{1+\cos s}{2}$ we get
\begin{equation*}
\phi(\mathbb{S}^n;r)= \frac{1}{nV_n}\displaystyle \sum_{k=0}^{\infty}\frac{(n)_k}{\left(\frac{n}{2}+1\right)_k}\frac{1}{2^k(k+1)}(1+\cos r)^{k+1}+C.
\end{equation*}
It is immediate to check that $\|p-q\|^2=2(1-\langle p, q \rangle)$ and $r=\arccos\langle p, q \rangle$. Hence,
$$
(1+\cos r)^{k+1}=2^{k+1}\left(1-\frac{\|p-q\|^2}{4}\right)^{k+1}, \quad p,q\in \mathbb{S}^n,
$$
and we thus get
\begin{equation*}
G(\mathbb{S}^n;p,q)=\frac{2}{nV_n}\displaystyle\sum_{k=0}^{\infty}\frac{(n)_k}{(k+1)\left(\frac{n}{2}+1\right)_k}\left(1-\frac{\|p-q\|^2}{4}\right)^{k+1}+C.
\end{equation*}
It remains to compute $C$ such that $G(\mathbb{S}^n;p,\cdot)$ has zero mean in $\mathbb{S}^n$. Without loss of generality we take $p=(0,\ldots,0,-1)\in\S^n$. Let $\varphi_{\mathbb{S}^n\backslash N}$ be the parameterization of $\mathbb{S}^n$ given by the inverse stereographic projection:
\begin{equation*}
\begin{array}{cccc}
\varphi:&\mathbb{R}^n & \to & \mathbb{S}^n\backslash (0,\ldots,0,1) \\
&z=\begin{pmatrix}
z_1, & \cdots, z_n
\end{pmatrix}
& \to &
\begin{pmatrix}
\displaystyle\frac{2z_1}{1+\|z\|^2}, & \cdots, & \displaystyle\frac{2z_n}{1+\|z\|^2}, & \displaystyle\frac{\|z\|^2-1}{1+\|z\|^2}
\end{pmatrix},
\end{array}
\end{equation*}
whose Jacobian equals $2^n/(1+\|z\|^2)^n$.  Then
$$
\|\varphi(0)-\varphi(z)\|^2=\frac{4\|z\|^2}{1+\|z\|^2},
$$
and so we have
\begin{multline*}
\frac{1}{V_n}\int_{\S^n}G(\mathbb{S}^n;p,q)dq\\
=C+\int_{\mathbb{R}^n}\frac{2}{nV_n^2}\displaystyle\sum_{k=0}^{\infty}\frac{(n)_k}{(k+1) \left(\frac{n}{2}+1\right)_k}\left(\frac{1}{1+\|z\|^2}\right)^{k+1} \frac{2^n}{(1+\|z\|^2)^{n}}dz,
\end{multline*}
which equals $0$ if and only if
$$
C=-\frac{2^{n+1}}{nV_n^2}\int_{\mathbb{R}^n}\displaystyle\sum_{k=0}^{\infty}\frac{(n)_k}{(k+1) \left(\frac{n}{2}+1\right)_k}\left(\frac{1}{1+\|z\|^2}\right)^{n+k+1}dz.
$$
It suffices to exchange the sum with the integral and pass to polar coordinates, recalling the definition of the beta function and the fact that
$$
V_n=2^{n-1}V_{n-1}B\left(\frac{n}{2},\frac{n}{2}\right),
$$
to conclude the result.
\end{proof}

\section{A lower bound for the Green energy on $\mathbb{S}^n$}
Recall that given $p_1,\ldots,p_N\in\S^n$, their Green energy is $E_{\mathbb{S}^n}(p_1, \ldots, p_N)=\sum_{i\neq j}G(\S^n;p_i,p_j)$.
The following quantity will play the role of the term ``$-1/2-\cot^2(a/2)\ln\cos (a/2)$'' in Proposition \ref{int_bolas_esfera}:
    $$K(\S^n,a)=\frac{2}{nV_n}\displaystyle\sum_{k=0}^{\infty} \frac{(n)_k}{\left(\frac{n}{2}+1\right)_k(k+1)}\frac{ B_{\sin^2\frac{a}{2}}(\frac{n}{2}+k+1,\frac{n}{2})}{B_{\sin^2\frac{a}{2}}(\frac{n}{2}, \frac{n}{2})},\quad a\in(0,\pi).$$
    Note that, from Proposition \ref{prop:greenSn},
    \begin{equation}\label{eq:GyK}
    G(\mathbb{S}^n;p,-p)=g(2)=-K(\mathbb{S}^n,\pi),\quad\forall p\in\mathbb{S}^n.
    \end{equation}
We state some basic properties of $K(\S^n,a)$. The first one just shows that $K(\S^2,a)$ equals $-1/2-\cot^2(a/2)\ln\cos (a/2)$ up to the multiplicative constant $-1/(2\pi)$, which makes sense in the light of \eqref{eq:GreenS2}.
\begin{lemma}\label{lem:Kabasic}
We have $K(\S^2,a)=-1/(2\pi)(-1/2-\cot^2(a/2)\ln\cos (a/2))$ for $a\in(0,\pi)$. Moreover,
    \begin{align*}
      K(\S^n,a)=&\frac{a^2}{(2n+4)V_n}+o(a^2),\\
      K(\S^n,\pi-a)=&K(\S^n,\pi)-\frac{a^2}{(2n-4)V_n}+o(a^2).
    \end{align*}
\end{lemma}
\begin{proof}
  See Section \ref{sec:LemaKbasico}
\end{proof}

Our main new result is the following lower bound for the Green energy on $\S^n$. Note that similar results are known for the Green energy in projective spaces \cite{Andersonetal}.
\begin{theorem}[Main result]\label{th:main}
Let $p_1,\ldots, p_N$ be $N$ points in $\mathbb{S}^n$ with $n\geq3$. Then,
\begin{align*}
E_{\mathbb{S}^n}(p_1, \ldots, p_N)\geq-\frac{n^{1+2/n}}{(n^2-4)V_n^{1-2/n}V_{n-1}^{2/n}}N^{2-2/n} +o(N^{2-2/n}),
\end{align*}
where, recall, $V_n$ is the volume of $\mathbb{S}^n$.
\end{theorem}
The proof of Theorem \ref{th:main} proceeds exactly as in Section \ref{sec:alternativeproofL}, changing $\S^2$ to $\S^n$:
\begin{enumerate}[(A)]
\item We write the analogous to Proposition \ref{int_bolas_esfera}, Lemma \ref{aux_cota_bolasnodisj} and Corollary \ref{cor:cotainferiormediaS2}: as before, let $B(p_0,a)$ be the geodesic ball centered at $p_0\in\mathbb{S}^n$ with radius $a>0$ and let $|B_a|$ be the volume of this ball.

\begin{prop}[$\S^n$--analogous of Proposition \ref{int_bolas_esfera}]\label{mediaGreenBola}
Let $p_0, p\in\mathbb{S}^n$ and $a\in(0,\pi)$. Then, $|B_a|=2^{n-1}V_{n-1}B_{\sin^2\frac{a}{2}}\left(\frac{n}{2},\frac{n}{2}\right)$ and moreover
\begin{enumerate}[i)]
\item If $p\notin B(p_0,a)$:
$$
\frac{1}{|B_a|}\int_{q\in B(p_0,a)}G(\mathbb{S}^n; p, q)dq=G(\S^n,p,p_0)+K(\S^n,a).
$$
\item If $p\in B(p_0,a)$:
$$
\frac{1}{|B_a|}\int_{q\in B(p_0,a)} G(\mathbb{S}^n; p, q)dq=-\frac{B_{\cos^2\frac{a}{2}}\left(\frac{n}{2},\frac{n}{2}\right)}{B_{\sin^2\frac{a}{2}} \left(\frac{n}{2},\frac{n}{2}\right)}\left(G(\mathbb{S}^n; p, -p_0)+K(\S^n,\pi-a)\right).
$$
\end{enumerate}
\end{prop}

\begin{proof}
See Section \ref{sec:proofofmediaGreenBola}.
\end{proof}

\begin{lemma}[$\S^n$--analogous of Lemma \ref{aux_cota_bolasnodisj}]\label{aux_cota_bolasnodisj_Sn}
The function $F:  [0,2]\times[0,\pi] \to  \mathbb{R}$ given by
\begin{equation*}
  F(t,\alpha)=  g(t)+K(\S^n,a)+ \frac{B_{\cos^2\frac{a}{2}}\left(\frac{n}{2},\frac{n}{2}\right)}{B_{\sin^2\frac{a}{2}} \left(\frac{n}{2},\frac{n}{2}\right)}\left(g\left(\sqrt{4-t^2}\right)+K(\S^n,\pi-a)\right),\end{equation*}
with $g(t)$ as in Proposition \ref{prop:greenSn}, is not negative.
\end{lemma}
\begin{proof}
  See Section \ref{sec:proofoflema4.2}.
\end{proof}
\begin{corollary}[$\S^n$--analogous of Corollary \ref{cor:cotainferiormediaS2}]\label{int_bolas_esfera_Sd}
Let $p_0, p\in\mathbb{S}^n$ and $a>0$. Then,
$$
\frac{1}{|B_a|}\int_{q\in B(p_0,a)} G(\S^n,p,q) dq\leq G(\S^n,p,p_0)+K(\S^n,a),
$$
with an equality if $p\notin B(p_0,a)$.
\end{corollary}
\begin{proof}
  Immediate from Proposition \ref{mediaGreenBola} and Lemma \ref{aux_cota_bolasnodisj_Sn}.
\end{proof}
\item We write the Green energy as follows:
\begin{align*}
E_{\mathbb{S}^n}(p_1, \ldots, p_N)=&\displaystyle\sum_{i\neq j} G(\mathbb{S}^n;p_i,p_j)\\
&-\frac{2N}{V_n}\displaystyle\sum_{i=1}^N \int_{q\in\mathbb{S}^n} G(\mathbb{S}^n;p_i,q)dq+\frac{N^2}{V_n^2}\int_{p,q\in\mathbb{S}^n}G(\mathbb{S}^n;p,q)dpdq.
\end{align*}
Note that from Definition \ref{Def_Energ_Green}, these two new integrals are equal to zero.
\item Let
\begin{align*}
U_{BB} &=\frac{N^2}{V_n^2}\int_{p,q\in\mathbb{S}^n}G(\mathbb{S}^n;p,q)dpdq=0,\\
U_i & = -\frac{2N}{V_n}\int_{\mathbb{S}^n}G(\mathbb{S}^n;p_i,q)dq=0,\\
U_{ij} & = G(\mathbb{S}^n;p_i,p_j),\\
\widehat{U}_i & = -\frac{2N}{V_n|B_a|}\int_{B(p_i,a)}\int_{\mathbb{S}^n} G(\mathbb{S}^n;p,q)dpdq=0,\\
\widehat{U}_{ij} & = \frac{1}{|B_a|^2} \int_{B(p_i,a)}\int_{B(p_j,a)} G(\mathbb{S}^n;p,q)dpdq.
\end{align*}
Define $\alpha,\beta,\gamma$ and $\delta$ by
\begin{align*}
E_{\mathbb{S}^n}(p_1, \ldots, p_N)=\underbrace{U_{BB}+\displaystyle\sum_{i=1}^N \widehat{U}_i+\displaystyle\sum_{i,j}\widehat{U}_{ij}}_{(\alpha)}+\underbrace{\displaystyle\sum_{i=1}^N (U_i-\widehat{U}_i)}_{(\beta)}\underbrace{-\displaystyle\sum_{i=1}^N \widehat{U}_{ii}}_{(\gamma)}+\underbrace{\displaystyle\sum_{i\neq j}(U_{ij} -\widehat{U}_{ij})}_{(\delta)}.
\end{align*}
Again $\alpha\geq0$ from Proposition \ref{positividad_alpha}, $\beta=0$ and as in Section \ref{sec:alternativeproofL}:
\begin{align*}
\gamma=&-\frac{N}{|B_a|^2} \int_{B(p_1,a)}\int_{B(p_1,a)} G(\S^n,p,q)dpdq\\
=&-\frac{N}{|B_a|^2}\int_{B(p_1,a)} \left(\int_{\S^2}G(\S^n,p,q)dp-\int_{B(-p_1,\pi-a)} G(\S^n,p,q)dp\right)dq
\\
=&\frac{N|B_{\pi-a}|}{|B_a|^2}\int_{B(p_1,a)} \left(
G(\S^n,-p_1,q)+K(\S^n,\pi-a)
\right)dq\\
=&\frac{N|B_{\pi-a}|}{|B_a|}\left(K(\S^n,\pi-a)+K(\S^n,a)+G(\S^n,-p_1,p_1)
\right)
\\
\stackrel{\eqref{eq:GyK}}{=}&\frac{N|B_{\pi-a}|}{|B_a|} \left(K(\S^n,\pi-a)+K(\S^n,a)-K(\S^n,\pi)
\right).
\end{align*}
We bound $\delta$ also following the same method as in Section \ref{sec:alternativeproofL}: for each $i\neq j$, from Corollary \ref{int_bolas_esfera_Sd} we have
\begin{align*}
U_{ij}-\widehat{U}_{ij}=& G(\mathbb{S}^n;p_i,p_j) - \frac{1}{|B_a|^2} \int_{B(p_i,a)}\int_{B(p_j,a)} G(\mathbb{S}^n;p,q)dpdq\\
\geq & G(\mathbb{S}^n;p_i,p_j) - \frac{1}{|B_a|} \int_{B(p_i,a)}G(\mathbb{S}^n;p_j,q)+K(\S^n,a) dq\\
\geq&-K(\S^n,a)+G(\mathbb{S}^n;p_i,p_j)-(G(\mathbb{S}^n;p_j,p_i)+K(\S^n,a))\\
=&-2K(\S^n,a).
\end{align*}
All in one, we have
$$
\delta\geq -2N(N-1)K(\S^n,a),
$$
and hence we conclude:
\begin{multline*}
E_{\mathbb{S}^n}(p_1, \ldots, p_N)\geq  -2N(N-1)K(\S^n,a)\\+\frac{N|B_{\pi-a}|}{|B_a|}\left(K(\S^n,\pi-a)+K(\S^n,a)-K(\S^n,\pi)
\right).
\end{multline*}
\item The inequality above is valid for any $a\in(0,\pi)$. Taking $a=C^{1/2}N^{-1/n}$ it is clear that, if $n\geq3$, a lower bound for $E_{\mathbb{S}^n}(p_1, \ldots, p_N)$ is
    $$
    (\star)=-2N^2K(\S^n,a)+\frac{nN^2V_n}{V_{n-1}C^{n/2}}\left(K(\S^n,\pi-a)+K(\S^n,a)-K(\S^n,\pi)
\right)+\text{l.o.t.}
    $$
    Using Lemma \ref{lem:Kabasic} we conclude that, if $n\geq3$,
    $$
    (\star)=-\frac{N^{2-2/n}}{n+2}\left(\frac{C}{V_n}+\frac{2nC^{1-n/2}}{(n-2)V_{n-1}}\right) +o(N^{2-2/n}).
    $$
    We are free to choose $C>0$. The optimal value $C=\left(\frac{nV_n}{V_{n-1}}\right)^{2/n}$  yields
        $$
    (\star)=-\frac{n^{1+2/n}}{(n^2-4)V_n^{1-2/n}V_{n-1}^{2/n}}N^{2-2/n} +o(N^{2-2/n}),
    $$
    finishing the proof of Theorem \ref{th:main}.
\end{enumerate}

In the light of \eqref{eq:cotasClog}, we see that the lower bound provided by the argument above in the case $n=2$ is surprisingly sharp. In the case of $\S^n$ for $n\geq3$ we have not found upper bounds on the Green energy to compare with, so we have performed some numerical experiments to provide some insight on the sharpness of our result. In Figure \ref{fig:comparaciones} we plot the (numerically found by a standard Riemannian-gradient method) minimal value of the Green energy for $n=4$ for an increasing number of points. This graphic suggests that our lower bound is again quite sharp in higher dimensions.

\begin{figure}
  \centering
  \includegraphics[width=\linewidth]{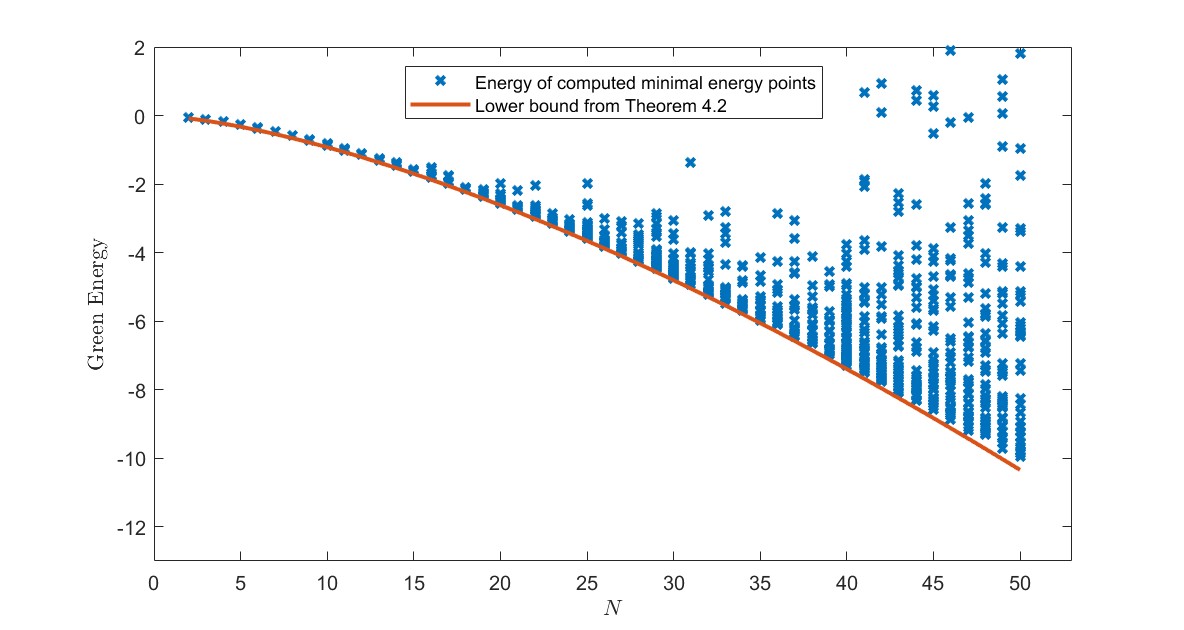}
  \caption{The minimum of the Green energy for increasing values of the number of points $N$, is compared to the lower bound provided by Theorem \ref{th:main} (solid line) in the case $n=4$. The minimal energy points have been obtained using Matlab's function {\tt{fminsearch}}, and crosses in the same vertical lines correspond to local minima fund by Matlab for different starting points.}\label{fig:comparaciones}
\end{figure}
\section{Proof of the technical results}\label{proofofballintegral}

\subsection{Proof of Lemma \ref{lem:Kabasic}}\label{sec:LemaKbasico}
If $n=2$, denoting $s=\sin^2\frac{a}{2}$ we have
$$K(\S^2,a)=\frac{1}{4\pi s}\displaystyle\sum_{k=0}^{\infty} \frac{s^{k+2}}{(k+1)(k+2)}.
$$
Now, the infinite sum is an analytic function in the complex unit disk and its second derivative is $\sum_{k\geq0}s^k=1/(1-s)$, and we deduce the value of the infinite sum above which is $s+(1-s)\log(1-s)$. All in one, we have
$$K(\S^2,a)=\frac{1}{4\pi}+\frac{\cos^2\frac{a}{2}\log\cos^2\frac{a}{2}}{4\pi \sin^2\frac{a}{2}},
$$
and the first claim is proved. For the second one, note that $K(\S^n,a)$ is complex analytic in the unit disk so we can just compute its Taylor series at $a=0$ term by term, and only the term $k=0$ has a nonzero second derivative at $a=0$ which yields
$$
K(\S^n,a)=\frac{2}{nV_n}\frac{ B_{\sin^2\frac{a}{2}}(\frac{n}{2}+1,\frac{n}{2})}{B_{\sin^2\frac{a}{2}}(\frac{n}{2}, \frac{n}{2})}+o(a^2)=\frac{a^2}{(2n+4)V_n}+o(a^2).
$$
The second asymptotic requires some more work. Denote
$$L(n) = \lim_{a\to0}\frac1{a^2}\left(K(\S^n,\pi)-K(\S^n,\pi-a)\right).$$ Using \eqref{eq:betacambia} it is clear that $L(n)=L_1(n)+L_2(n)$ where
    \begin{align*}
      L_1(n) = & \frac{-2}{nV_n}\lim_{a\to0}\frac1{a^2} \sum_{k=0}^{\infty} \frac{(n)_k}{\left(\frac{n}{2}+1\right)_k(k+1)}\frac{ B_{\sin^2\frac{a}{2}} B(\frac{n}{2},\frac{n}{2})B\left(\frac{n}{2}+k+1,\frac{n}{2}\right)}{B_{\cos^2\frac{a}{2}}(\frac{n}{2}, \frac{n}{2}) B(\frac{n}{2}, \frac{n}{2})},\\
      L_2(n) = & \frac{2}{nV_n}\lim_{a\to0}\frac1{a^2} \sum_{k=0}^{\infty} \frac{(n)_k}{\left(\frac{n}{2}+1\right)_k(k+1)}\frac{ B_{\sin^2\frac{a}{2}}(\frac{n}{2},\frac{n}{2}+k+1)}{B_{\cos^2\frac{a}{2}}(\frac{n}{2}, \frac{n}{2})}.
    \end{align*}
    The elementary upper bound from Lemma \ref{cota_beta} implies that $L_1(n)=0$ if $n\geq3$. As for $L_2(n)$, the asymptotic $a/2\sim \sin a/2$, the definition of the incomplete Beta function, the Monotone Convergence Theorem and the change of variables $t=u\sin^2\frac{a}{2}$ successively yield
    \begin{align*}
      L_2(n)
      =& \lim_{a\to0}\sum_{k=0}^{\infty}\frac1{\sin^2\frac{a}2}\int_0^{\sin^2\frac{a}{2}} \frac{t^{\frac{n}{2}-1} (1-t)^{\frac{n}{2}-1}}{2nV_nB(\frac{n}{2}, \frac{n}{2})} \frac{(n)_k(1-t)^{k+1}}{\left(\frac{n}{2}+1\right)_k(k+1)}dt\\
       =& \lim_{a\to0}\frac1{\sin^2\frac{a}2}\int_0^{\sin^2\frac{a}{2}} \frac{t^{\frac{n}{2}-1} (1-t)^{\frac{n}{2}-1}}{2nV_nB(\frac{n}{2}, \frac{n}{2})}\sum_{k=0}^{\infty} \frac{(n)_k(1-t)^{k+1}}{\left(\frac{n}{2}+1\right)_k(k+1)}dt\\
       =&\lim_{a\to0}\int_0^{1} \frac{(u\sin^2\frac{a}{2})^{\frac{n}{2}-1} (1-u\sin^2\frac{a}{2})^{\frac{n}{2}-1}}{2nV_nB(\frac{n}{2},\frac{n}{2})} \sum_{k=0}^{\infty} \frac{(n)_k(1-u\sin^2\frac{a}{2})^{k+1}}{\left(\frac{n}{2}+1\right)_k(k+1)}du.
    \end{align*}
    From the upper bound of Lemma \ref{exp_sum_cota}, the inner function is bounded above by some constant $C(n)$ whose value we do not need to know. Hence, we can interchange limit and integral by Lebesgue's Dominated Convergence Theorem getting
        \begin{align*}
      L_2(n)
       =&\int_0^{1} \frac{1}{2nV_nB(\frac{n}{2}, \frac{n}{2})}\lim_{a\to0}\left(\left(u\sin^2\frac{a}{2}\right)^{\frac{n}{2}-1} \sum_{k=0}^{\infty} \frac{(n)_k(1-u\sin^2\frac{a}{2})^{k+1}}{\left(\frac{n}{2}+1\right)_k(k+1)}\right)du.
    \end{align*}
    Finally, from Lemma \ref{lem:asyntotic}, the inner limit is constantly equal to $\Gamma(\frac{n}{2}+1)\Gamma(\frac{n}{2}-1)/\Gamma(n)$, which yields
     \begin{align*}
      L_2(n)=&\frac{\Gamma(\frac{n}{2}+1)\Gamma(\frac{n}{2}-1)}{2nV_nB(\frac{n}{2}, \frac{n}{2})\Gamma(n)}=\frac{1}{(2n-4)V_n},
      \end{align*}
      and the lemma follows.

\subsection{Proof of Proposition \ref{mediaGreenBola}}\label{sec:proofofmediaGreenBola}

Let $S(p_0,r)\subset \mathbb{S}^n$ be a geodesic sphere contained in  $\mathbb{S}^n$ and $|S_r|$ its volume. Since $S(p_0,r)$ is a $(n-1)$--dimensional sphere of radius $\sin r$, its volume is $\text{vol}( S(p_0,r))=V_{n-1}\sin^{n-1}r$. The volume of the geodesic ball is then (using Lemma \ref{lem:betas}):
$$
|B_a|=\int_0^a\text{vol}( S(p_0,r))\,dr=V_{n-1}\int_0^a\sin^{n-1}r\,dr= 2^{n-1}V_{n-1}B_{\sin^2\frac{r}{2}}\left(\frac{n}{2},\frac{n}{2}\right).
$$

The mean value equality if $p\not\in B(p_0,a)$ is essentially a known result (at least the existence of the constant $K(\S^n,a)$, although its exact value seems not to be present in the literature): consider the function $H(q)=G(\S^n,p,q)-G(\S^n,-p_0,q)$. Then, since $p,-p_0\notin B(p_0,a)$, we conclude that $\Delta H\equiv0$ in that ball. From \cite[Theorem 1]{Willmore}, the mean value of $H$ on $S(p_0,r)$ with $r<a$ is equal to $H(p_0)=G(\S^n,p,p_0)-G(\S^n,-p_0,p_0)$. That is,

\begin{multline*}
\int_{q\in S(p_0,r)}(G(\mathbb{S}^n; p, q)-G(\S^n,p,p_0))dq
=\\
\int_{q\in S(p_0,r)}(G(\mathbb{S}^n; -p_0, q)-G(\S^n,-p_0,p_0))dq.
\end{multline*}
The integral of $G(\mathbb{S}^n; p, q)-G(\S^n,p,p_0)$ when $q\in B(p_0,a)$ then equals
\begin{multline*}
\int_0^a\int_{q\in S(p_0,r)}(G(\mathbb{S}^n; p, q)-G(\S^n,p,p_0))dq\,dr\\
=\int_0^a\int_{q\in S(p_0,r)}(G(\mathbb{S}^n; -p_0, q)-G(\S^n,-p_0,p_0))dq\,dr.
\end{multline*}
In other words, we can assume that $p=-p_0$ to finish the mean value computation.  Now this is a straightforward computation: denote $$A=|B_a|^{-1}\int_{q\in B(p_0,a)}(G(\mathbb{S}^n; -p_0, q))dq-G(\mathbb{S}^n; -p_0, p_0)$$ and note that
\begin{align*}
 A = &  \frac{1}{|B_a|}\int_0^a\int_{q\in S(p_0,r)}G(\mathbb{S}^n; -p_0, q)dq-G(\mathbb{S}^n; -p_0, p_0)\\
  = & \frac{1}{|B_a|}\int_0^a \text{vol}(S(p_0,r))\frac{2}{nV_n}\displaystyle\sum_{k=0}^{\infty} \frac{(n)_k}{(k+1)\left(\frac{n}{2}+1\right)_k} \sin^{2k+2}\frac{r}{2}\,dr
  \\
  = & \frac{2}{2^{n-1}nV_nB_{\sin^2\frac{a}{2}}\left(\frac{n}2,\frac{n}2\right)}\displaystyle\sum_{k=0}^{\infty} \frac{(n)_k}{(k+1)\left(\frac{n}{2}+1\right)_k} \int_0^a\sin^{n-1}r\sin^{2k+2}\frac{r}{2}\,dr
  \\
  = & \frac{2}{nV_nB_{\sin^2\frac{a}{2}}\left(\frac{n}2,\frac{n}2\right)}\displaystyle\sum_{k=0}^{\infty} \frac{(n)_k}{(k+1)\left(\frac{n}{2}+1\right)_k} \int_0^a\sin^{2k+n+1}\frac{r}{2}\cos^{n-1}\frac{r}{2}\,dr
  \\
 =& \frac{2}{nV_nB_{\sin^2\frac{a}{2}}\left(\frac{n}2,\frac{n}2\right)}\displaystyle\sum_{k=0}^{\infty} \frac{(n)_k}{(k+1)\left(\frac{n}{2}+1\right)_k} \int_0^{\sin^2\frac{a}{2}}u^{k+\frac{n}{2}}(1-u)^{\frac{n}{2}-1}\,du
 \\
 =& \frac{2}{nV_nB_{\sin^2\frac{a}{2}}\left(\frac{n}2,\frac{n}2\right)}\displaystyle\sum_{k=0}^{\infty} \frac{(n)_k}{(k+1)\left(\frac{n}{2}+1\right)_k} B_{\sin^2\frac{a}{2}}\left(k+\frac{n}{2}+1,\frac{n}{2}\right)
 \\
 =&K(\S^n,a),
\end{align*}
as claimed.

Now, suppose that $p\in B(p_0,a)$. Then,
\begin{align*}
\int_{q\in B(p_0,a)}G(\mathbb{S}^n; p, q)dq &=\int_{q\in \mathbb{S}^n}G(\mathbb{S}^n; p, q)dq-\int_{q\in B(-{p}_0,\pi-a)}G(\mathbb{S}^n; p, q)dq\\
&=-\int_{q\in B(-p_0,\pi-a)}G(\mathbb{S}^n; p, q)dq
\end{align*}
since $G(\mathbb{S}^n; p,\cdot)$ has zero mean for all $p\in\mathbb{S}^n$. From the previous claim we have:
$$
\int_{q\in B(-p_0,\pi-a)}  G(\mathbb{S}^n; p, q)dq=|B_{\pi-a}|
\left(G(\mathbb{S}^n; p, -p_0)+K(\S^n,\pi-a)\right),
$$
which readily gives the second item.

\subsection{Proof of Lemma \ref{aux_cota_bolasnodisj_Sn}}\label{sec:proofoflema4.2}
Carefully collecting the terms and writing down the definition of $g(t)$ and $K(\S^n,a)$, Lemma \ref{aux_cota_bolasnodisj_Sn} is implied by the following.

\begin{prop}\label{cota_para_delta}
The function $F:[0,1]\times[0,1]$ given by
\begin{align*}
F(s,\alpha)&= \displaystyle\sum_{k=0}^{\infty}\frac{(n)_k}{(\frac{n}{2}+1)_k(k+1)}Q_k(s,\alpha),\text{ where }\\
Q_k(s,\alpha) &=(1-s)^{k+1}B_{\alpha}\left(\frac{n}{2},\frac{n}{2}\right)+s^{k+1} B_{1-\alpha}\left(\frac{n}{2},\frac{n}{2}\right)\\
&+B_{\alpha}\left(\frac{n}{2}+k+1,\frac{n}{2}\right)-B_{\alpha} \left(\frac{n}{2},\frac{n}{2}+k+1\right),
\end{align*}
is not negative.
\end{prop}

\begin{proof}
We first fix $s$ and look for minima in $Q_k(s,\alpha)$ with respect to $\alpha$. We have
$$
\frac{\partial}{\partial \alpha} Q_k(s,\alpha)=\alpha^{\frac{n}{2}-1}(1-\alpha)^{\frac{n}{2}-1}[(1-s)^{k+1}-s^{k+1}+\alpha^{k+1}-(1-\alpha)^{k+1}],
$$
which readily implies that $Q_k(s,\alpha)$ is a decreasing function of $\alpha$ in $[0,s]$ and a increasing function of $\alpha$ in $[s,1]$. That is, $Q(s,\alpha)$ has a global minimum at $\alpha=s$ and consequently so does $F(s,\alpha)$. Hence it suffices to check that $F(s,s)\geq0,\,\forall s$. From \eqref{def_incompletebeta} and \cite[p. 1008, (9.131.1)]{GR15}, we have
\begin{align*}
\frac{\partial}{\partial s}F(s,s)=& \sum_{k=0}^{\infty} \frac{(n)_k}{(n/2+1)_k}\left(s^kB_{1-s}\left(\frac{n}{2},\frac{n}{2}\right)-(1-s)^kB_s\left(\frac{n}{2}, \frac{n}{2}\right)\right)\\
=& B_{1-s}\left(\frac{n}{2},\frac{n}{2}\right){}_2F_1(1,n;n/2+1;s)-B_s\left(\frac{n}{2},\frac{n}{2}\right){}_2F_1(1,n;n/2+1;1-s)\\
=& 0.
\end{align*}
We will prove that $\lim_{s\to 0}F(s,s)=0$ to conclude the result. To this end, we define
\begin{align*}
T_1(s) &= B_s\left(\frac{n}{2},\frac{n}{2}\right)\displaystyle\sum_{k=0}^{\infty}\frac{(n)_k}{ \left(\frac{n}{2}+1\right)_{k}(k+1)}(1-s)^{k+1},\\
T_2(s) &= B_{1-s}\left(\frac{n}{2},\frac{n}{2}\right)\displaystyle\sum_{k=0}^{\infty}\frac{(n)_k}{ \left(\frac{n}{2}+1\right)_{k}(k+1)}s^{k+1},\\
T_3(s) &= \displaystyle\sum_{k=0}^{\infty}\frac{(n)_k}{\left(\frac{n}{2}+1\right)_{k}(k+1)}B_s\left( \frac{n}{2}+k+1,\frac{n}{2}\right),\\
T_4(s) &= \displaystyle\sum_{k=0}^{\infty}\frac{(n)_k}{\left(\frac{n}{2}+1\right)_{k}(k+1)}B_s\left( \frac{n}{2},\frac{n}{2}+k+1\right).
\end{align*}
So $F(s,s)=T_1(s)+T_2(s)+T_3(s)-T_4(s)$, and it suffices to check that each $T_i(s)\to0$ as $s\to0$, which is quite immediate from lemmas \ref{cota_beta} and \ref{exp_sum_cota}:
\begin{align*}
T_1(s)\leq& \,C(n)s^{\frac{n}{2}}S_n(1-s)\leq C(n)s,\\
T_2(s)\leq& \,C(n)S_n(s) \leq \frac{C(n)s}{(1-s)^{\frac{n}{2}-1}},\\
T_3(s)\leq& \sum_{k=0}^{\infty}\frac{(n)_k}{\left(\frac{n}{2}+1\right)_{k}(k+1)}s^{\frac{n}{2}+k+1}=s^{\frac{n}{2}}S_n(s)\leq \frac{s^{\frac{n}{2}+1}C(n)}{(1-s)^{\frac{n}{2}-1}},\\
T_4(s)\leq & \int_0^s t^{\frac{n}{2}-1}(1-t)^{\frac{n}{2}-1}S_n(1-t)dt
\leq  \,C(n)\int_0^s (1-t)^{n/2}dt\leq C(n)s,
\end{align*}
where each apparition of $C(n)$ may be a different constant. This finishes the proof.

\end{proof}


\appendix


\section{Some special functions}\label{Func_esp}

The contents of this appendix are standard and are mainly taken from \cite{BHS19} and \cite{GR15}. Recall the Pochhammer symbol, defined for $n\in\C$ and $k\geq0$ an integer:
\begin{equation}\label{eq:pochhammer}
(n)_k=n(n+1)\cdots (n+k-1), \quad (n)_0=1.
\end{equation}
Clearly, $(1)_k=k!$, and if $n,n+k$ are not zero or negative integers, then
$$
(n)_k=\frac{\Gamma(n+k)}{\Gamma(n)}.
$$

\subsection{The Gaussian hypergeometric series}

Recall that the Gaussian hypergeometric series is analytic in $|z|<1$ and defined by the power series
\begin{equation}\label{Serie_hiper_Gauss}
{}_2F_1(a,b;c;z)=\displaystyle\sum_{k=0}^{\infty}\frac{(a)_k(b)_k}{(c)_k}\frac{z^k}{k!}.
\end{equation}
It is undefined for integer $c\leq0$.

\subsection{The beta function}

Euler's beta function is
\begin{equation*}
B(\alpha,\beta)=\int_0^1 t^{\alpha-1}(1-t)^{\beta-1}dt=\int_0^{+\infty} \frac{t^{\alpha-1}}{(1+t)^{\alpha+\beta}}, \quad \text{Re}(\alpha), \,\text{Re}(\beta)>0,
\end{equation*}
and clearly satisfies $B(\alpha,\beta)=B(\beta,\alpha)$. The incomplete beta function
\begin{equation}\label{def_incompletebeta}
B_x(\alpha,\beta)=\int_0^x t^{\alpha-1}(1-t)^{\beta-1}dt=\frac{x^{\alpha}}{\alpha}{}_2F_1(\alpha,1-\beta;\alpha+1;x),
\end{equation}
satisfies
\begin{equation}\label{eq:betacambia}
B(\alpha,\beta)=B_x(\alpha,\beta)+B_{1-x}(\beta, \alpha).
\end{equation}

\begin{lemma}\label{lem:betas}
For all $r\in[0,\pi]$ we have
\begin{align*}
\int_0^r\sin^{n-1}t\,dt=&2^{n-1}B_{\sin^2\frac{r}{2}}\left(\frac{n}{2},\frac{n}{2}\right), \\
\int_{r}^{\pi} \sin^{n-1} t\, dt=&2^{n-1}B_{\cos^2\frac{r}{2}}\left(\frac{n}{2},\frac{n}{2}\right).
\end{align*}
\end{lemma}
\begin{proof}
In both cases, the functions involved have the same value at $r=0$ or $r=\pi$ respectively, they are smooth and their derivatives coincide, hence they are equal.

\end{proof}

\begin{lemma}\label{cota_beta}
For $0<s<1$ and $\alpha,\beta\in(0,\infty)$ we have:
$$
\frac{(1-s)^{\beta-1}s^{\alpha}}{\alpha}=(1-s)^{\beta-1}\int_0^s t^{\alpha-1} \leq B_s(\alpha,\beta) \leq \int_0^s t^{\alpha-1}dt=\frac{s^{\alpha}}{\alpha}.
$$
\end{lemma}

\section{The Green function in a compact Riemannian manifold}\label{Funcion_Green}
The Green function is a natural potential on any compact manifold.

\begin{definition}\label{Def_Energ_Green}
The Green energy of $N$ points $p_1,\ldots, p_N$ on a compact Riemannian manifold $\mathcal{M}$ is
$$
E_{\mathcal{M}}(p_1,\ldots,p_N)=\displaystyle\sum_{i\neq j} G(\mathcal{M};p_i,p_j),
$$
where $G:\mathcal{M}\times\mathcal{M}\setminus\{(p,p):p\in\mathcal{M}\} \to \mathbb{R}$ is the unique function, known as Green function (potential), with the following properties:
\begin{enumerate}
\item $\Delta_q G=\mathcal{S}_p(q)-\text{vol}(\mathcal{M})^{-1}$, with $\mathcal{S}_p$ the Dirac's delta function, in the sense of distributions.
\item Symmetry: $G(\mathcal{M};p,q)=G(\mathcal{M};q,p)$.
\item $G(\mathcal{M};p,\cdot)$ has mean zero $\forall p\in\mathcal{M}$, i.e., $\int_{q\in\mathcal{M}}G(\mathcal{M};p,q)dq=0$.
\end{enumerate}
\end{definition}
We follow the convention that the Riemannian Laplacian is given by $\Delta=-\text{div} \nabla$.
The points that minimize the Green energy are asymptotically uniformly distributed in any compact Riemannian manifold (see the main result in \cite{BeltranCorralCriado2019}).

\begin{prop}\label{positividad_alpha}
Let $G(\mathcal{M};p,q)$ be the Green function of a manifold $\mathcal{M}$ and $\nu$ any finite signed measure such that $\nu(M)=0$, i.e., $\int_{\mathcal{M}}\nu\text{dvol}=0$. Then,
\begin{equation*}
\int_{p,q\in\mathcal{M}} G(\mathcal{M};p,q)d\nu(p)d\nu(q)\geqslant 0,
\end{equation*}
with an equality if and only if $\nu=0$.
\end{prop}

\begin{proof}
See \cite[p. 166, Definition 3.2]{BeltranCorralCriado2019} and \cite[p. 175, Proposition 3.14]{BeltranCorralCriado2019}.
\end{proof}

\subsection{Computation of the Green function in compact harmonic manifolds}\label{AlgoritmoFGreen}

The Green function of a general Riemannian manifold can be very hard to compute. In \cite{BeltranCorralCriado2019}, a method is given to compute it in compact harmonic manifolds, that is to say, in spheres and projective spaces (see \cite{Andersonetal} for an alternative, equivalent method). We have used it to get the explicit expression of $G(\mathbb{S}^n;p,q)$ in Proposition \ref{prop:greenSn}.

\smallskip

Given a compact harmonic manifold $\mathcal{M}$, its Green function $G(\mathcal{M};p,q)$ is given by $G(\mathcal{M};p,q)=\phi(\mathcal{M};d_R(p,q))=\phi(\mathcal{M};r)$ for all $p,q\in\mathcal{M}$, where $\phi(\mathcal{M};r)$ satisfies
$$
\phi'(\mathcal{M};r)=-\frac{\int_r^D t^{n-1}\Omega(t)dt}{Vr^{n-1}\Omega(r)}.
$$
Here, $V$ is the volume of $\mathcal{M}$, $D$ its diameter, $r=d_R(p,q)$ the Riemannian distance and $\Omega(t)$ is the volume density function, that in the case of the sphere $\S^n$ satisfies
$$r^{n-1}\Omega(r)=\sin^{n-1}r.$$
(See \cite{BeltranCorralCriado2019} for the density functions of the projective spaces, and note that there are some ambiguities in the choice of the normalization that produce constant factors as powers of $2$: these constants do not affect the computation of the Green function).
\smallskip

\section{A bound on a function given by its series}
We have used several times the function
\begin{equation}\label{function_Sn}
S_n(s)=\sum_{k=0}^{\infty} \frac{(n)_k}{(\frac{n}{2}+1)_k(k+1)}s^{k+1}
=s+\sum_{k=1}^{\infty} \frac{(n)_k}{(\frac{n}{2}+1)_k(k+1)}s^{k+1},\quad s\in[0,1).
\end{equation}
In this section we show some elementary bounds and asymptotics regarding $S_n(s)$.
Recall Stirling's formula, valid for $x>0$:
$$
\sqrt{\frac{2\pi}{x}}\left(\frac{x}{e}\right)^x\leq \Gamma(x)\leq \sqrt{\frac{2\pi}{x}}\left(\frac{x}{e}\right)^x e^{\frac{1}{12x}},
$$
from which we immediately deduce that if $k\geq 1,\,\,n\geq 2$:
\begin{equation}\label{eq:stirling}
\frac{(n)_k}{(\frac{n}{2}+1)_k(k+1)}=
\frac{\Gamma(n+k)\Gamma(\frac{n}{2}+1)}{\Gamma(n)\Gamma(\frac{n}{2}+k+1)(k+1)} =\frac{\Gamma(\frac{n}{2}+1)}{\Gamma(n)}k^{\frac{n}{2}-2} +O(k^{\frac{n}{2}-3}).
\end{equation}

\begin{lemma}\label{exp_sum_cota}
Let $p>-1$. Then
\begin{equation*}
\displaystyle\sum_{k=1}^{\infty} k^{p}s^{k+1}\leq{\frac{sC(p)}{(1-s)^{p+1}}},\quad s\in[0,1),
\end{equation*}
for some constant $C(p)$ depending only on $p$. In particular, for $n\geq3$,
$$
S_n(s)\leq\frac{sC(n)}{(1-s)^{\frac{n}2-1}},
$$
where $S_n(s)$ is given in \eqref{function_Sn}.
\end{lemma}

\begin{proof}
There is no harm in assuming $s>0$. Let $f(x)=x^{p}s^{x+1}$. If $p\in(-1,0]$ then $f(x)$ is a decreasing function and
$$
\displaystyle\sum_{k=1}^{\infty} k^{p}s^{k+1}\leq\int_0^\infty f(x)\,dx=\frac{s\Gamma(p+1)}{(\ln\frac1s)^{p+1}}\leq\frac{sC(p)}{(1-s)^{p+1}}.
$$
If $p>0$, $f(x)$ attains its global maximum
$$s\left(\frac{p}{e\ln \frac1s}\right)^{p}$$ at $x_{\text{max}}= p/(\ln \frac1s)$. The comparison of the sum with the integral must be done in the two intervals separated by this point since in one side the terms are increasing and in the other side they are decreasing. All in one, we have
\begin{align*}
\sum_{k=1}^{\infty} k^{p}s^{k+1}\leq &3s\left(\frac{p}{e\ln \frac1s}\right)^{p}+\int_1^{x_{\text{max}}} x^{p}s^{x+1}\,dx+\int_{x_{\text{max}}}^\infty x^{p}s^{x+1}\,dx\\
\leq &\frac{sC(p)}{\left(\ln \frac1s\right)^{p}}+\int_0^{\infty} x^{p}s^{x+1}\,dx\\
=&\frac{sC(p)}{\left(\ln \frac1s\right)^{p}}+\frac{s\Gamma(p+1)}{(\ln\frac1s)^{p+1}},
\end{align*}
and the lemma follows. The same comparison argument also yields
$$
\sum_{k=1}^{\infty} k^{p}s^{k+1}\geq -\frac{sC(p)}{\left(\ln \frac1s\right)^{p}}+\frac{s\Gamma(p+1)}{(\ln\frac1s)^{p+1}}.
$$
\end{proof}

A finer analysis of the approximation argument above unveils the asymptotics of $S_n$ as its argument approaches $1$:
\begin{lemma}\label{lem:asyntotic}
The following asymptotic equality holds for $n\geq3$:
$$
\lim_{s\to 0}s^{\frac{n}{2}-1}S_n(1-s)= \frac{\Gamma(\frac{n}{2}+1)\Gamma(\frac{n}{2}-1)}{\Gamma(n)}.
$$
\end{lemma}
\begin{proof}
 From \eqref{eq:stirling} we have for some constant $C(n)>0$:
  \begin{align*}
  s^{\frac{n}{2}-1}S_n(1-s)\stackrel{\eqref{eq:stirling}}{\leq}
  \frac{\Gamma(\frac{n}{2}+1)}{\Gamma(n)}s^{\frac{n}{2}-1} \sum_{k=0}^{\infty} k^{\frac{n}{2}-2}\left(1+\frac{C(n)}{k}\right) (1-s)^{k+1}.
  \end{align*}
  Now, Lemma \ref{exp_sum_cota} implies that the sum with the $C(n)/k$ term is bounded above by a constant divided by $s^{\frac{n}{2}-2}$, which becomes irrelevant in the limit. Hence,
  \begin{align*}
  \lim_{s\to 0}s^{\frac{n}{2}-1}S_n(1-s)\leq&\lim_{s\to 0}\frac{\Gamma(\frac{n}{2}+1)s^{\frac{n}{2}-1}}{\Gamma(n)}
  \int_1^{\infty} x^{\frac{n}{2}-2}(1-s)^{x+1}\,dx\\
  =&\lim_{s\to 0}\frac{\Gamma(\frac{n}{2}+1)s^{\frac{n}{2}-1}\Gamma(\frac{n}{2}-1)}{\Gamma(n) \ln\left(\frac{1}{1-s}\right)^{\frac{n}{2}-1}}\\
  =&\frac{\Gamma(\frac{n}{2}+1)\Gamma(\frac{n}{2}-1)}{\Gamma(n)}.
  \end{align*}
  The lower bound is proved the same way (this time the constant $C(n)$ is negative, but again it plays no role in the limit).
\end{proof}


\end{document}